\documentclass{article}
\usepackage[utf8]{inputenc}
\usepackage{graphicx}
\usepackage{amsfonts}
\usepackage{amsmath}
\usepackage{amssymb}
\usepackage{fancyhdr}
\usepackage{titlesec}
\usepackage{indentfirst}
\usepackage{booktabs}
\usepackage{verbatim}
\usepackage{color}
\usepackage{amsthm}
\usepackage{subfigure}
\usepackage{mathabx}

\usepackage[page,header]{appendix}
\usepackage{titletoc}

\newcommand{\rd}{\,\mathrm{d}}
\numberwithin{equation}{section}
\newtheorem{theorem}{Theorem}[section]
\newtheorem{lemma}[theorem]{Lemma}
\newtheorem{corollary}[theorem]{Corollary}

\newtheorem{proposition}[theorem]{Proposition}

\newtheorem{remark}[theorem]{Remark}

\usepackage{mathtools}

\def\bx{{\bf x}}
\def\by{{\bf y}}

\def\bv{{\bf v}}

\def\cM{\mathcal{M}}

\def\cP{\mathcal{P}}
\def\cQ{\mathcal{Q}}

\def\supp{\textnormal{supp\,}}

\def\dist{\textnormal{dist\,}}

\begin{document}

\title{A family of explicit minimizers for interaction energies}

\author{ Ruiwen Shu\footnote{Department of Mathematics, University of Georgia, Athens, GA 30602 (ruiwen.shu@uga.edu).}}

\maketitle

\abstract{In this paper we consider the minimizers of the interaction energies with the power-law interaction potentials $W(\bx) = \frac{|\bx|^a}{a} - \frac{|\bx|^b}{b}$ in $d$ dimensions. For odd $d$ with $(a,b)=(3,2-d)$ and even $d$ with $(a,b)=(3,1-d)$, we give the explicit formula for the unique energy minimizer up to translation. For the odd dimensions, the key observation is that successive Laplacian of the Euler-Lagrange condition gives a local partial differential equation for the minimizer. For the even dimensions $d$, the minimizer is given as the projection and rescaling of the previously constructed minimizer in dimension $d+1$ via a new lemma on dimension reduction.}

\section{Introduction}

In this paper we consider the interaction energy
\begin{equation}\label{E}
	E[\rho] = \frac{1}{2}\int_{\mathbb{R}^d}\int_{\mathbb{R}^d}W(\bx-\by)\rd{\rho(\by)}\rd{\rho(\bx)}\,,
\end{equation}
where $\rho\in \cM(\mathbb{R}^d)$ is a probability measure. $W:\mathbb{R}^d\rightarrow \mathbb{R}\cup\{\infty\}$ is the attractive-repulsive power-law interaction potential given by
\begin{equation}\label{Wab}
	W(\bx) = \frac{|\bx|^a}{a} - \frac{|\bx|^b}{b}\,,
\end{equation}
where $-d < b < a$. Here we adopt the tradition that $\frac{|\bx|^0}{0}$ is understood as $\ln|\bx|$.

The existence of minimizers of $E$ with $W$ given by \eqref{Wab} is known \cite{SST15,CCP15,CFT15}. The uniqueness of minimizers is also studied by many recent works \cite{Lop19,CS21,DLM1,DLM2,Fra}, mainly using convexity theory. We say $W$ has the linear interpolation convexity (LIC), if for any distinct compactly-supported $\rho_0,\rho_1\in\cM(\mathbb{R}^d)$ such that $\int_{\mathbb{R}^d}\bx\rho_0(\bx)\rd{\bx}=\int_{\mathbb{R}^d}\bx\rho_1(\bx)\rd{\bx}$, the function $t\mapsto E [(1-t)\rho_0+t\rho_1]$ is strictly convex on $t\in [0,1]$. For LIC potentials, the energy minimizer is necessarily unique (up to translation, which will be omitted later when it is clear from the context), and the LIC property is closely related to the positivity of $\hat{W}$, the Fourier transform of $W$. Also, if $W$ is radial and LIC, then its minimizer is necessarily radially symmetric. For the power-law potentials \eqref{Wab}, it is known \cite{Lop19,CS21,Fra} that LIC holds when $-d<b\le 2,\,2\le a \le 4,\,b<a$ except for $(a,b)=(4,2)$ which satisfies a non-strict version of LIC \cite{DLM2}.

Convexity theory is also crucial for deriving the explicit formula for minimizers. In fact, the Euler-Lagrange condition for minimizers, see \eqref{prop_qual_1} below, is only a necessary condition, but it becomes also sufficient if the interaction potential is LIC \cite{CS21}. Therefore, for LIC potentials, one can obtain explicit minimizers by verifying the Euler-Lagrange condition, which only requires  explicit calculation. 

The known explicit formulas for the minimizers for power-law potentials can be categorized as follows:
\begin{itemize}
	\item For $-d<b<\min\{4-d,2\}$ and $a=2$, the unique minimizer is
	\begin{equation}\label{Rx2}
		\rho(\bx) = C(R^2-|\bx|^2)_+^{1-\frac{b+d}{2}},\quad R = \Big(\frac{\Gamma(\frac{4-b}{2})\Gamma(\frac{b+d}{2})}{\Gamma(\frac{2+d}{2})}\Big)^{\frac{1}{2-b}}\,,
	\end{equation}
	and $C$ is a normalizing constant. The special case $b=2-d$ (Newtonian repulsion) was a classical result \cite{Frost}. The cases $-d<b<2-d$ were due to \cite{CV11}. For the cases $2-d<b<\min\{4-d,2\}$, this family of functions, as steady states of the energy, were obtained by \cite{CH17}. They were shown to be the minimizer by \cite{CS21,Fra}. \cite{CH17} also constructed steady states for the same range of $b$ and any positive even integer $a$, but it is not known whether they are energy minimizers.
	\item For $d\ge 2$, $2\le a \le 4$, $\frac{-10+3a+7d-a d-d^2}{d+a-3}=:b_*(a) \le b \le 2$, $b<a$, the unique minimizer is the uniform distribution on the sphere with radius
	\begin{equation}
		R = \frac{1}{2}\Big(\frac{\Gamma(\frac{d+b-1}{2})\Gamma(\frac{2d+a-2}{2})}{\Gamma(\frac{d+a-1}{2})\Gamma(\frac{2d+b-2}{2})}\Big)^{\frac{1}{a-b}}\,,
	\end{equation}
	(except for $(a,b)=(4,2)$, as discussed below). The uniform distribution on a sphere, as steady states, was first studied by \cite{BCLR13_2} by local stability analysis. They are shown to be the minimizer by \cite{FM} (see also earlier partial results in \cite{DLM1}). It is also known that these are all the power-law potentials such that LIC holds and the minimizer is the uniform distribution on a sphere \cite{FM}.
	\item For $(a,b)=(4,2)$, there are infinitely many minimizers, given by any $\rho\in\cM(\mathbb{R}^d)$ supported on a sphere of radius $\sqrt{\frac{d}{2d+2}}$, satisfying $\int_{\mathbb{R}^d}x_ix_j\rd{\rho(\bx)} = \frac{d}{2d+2}\delta_{ij}$. This was obtained by \cite{DLM2}.
	\item For $a\ge \min\{2+d,4\}$, $b\ge 2$, $b<a$, the unique minimizer (up to translation and rotation) is the sum of equal Dirac masses, located at the vertices of a unit simplex. This was obtained by \cite{DLM2}.
	\item For $d=1$, $b=2$, $2<a<3$, the unique minimizer is
	\begin{equation}
		\rho(x) = C(R^2-x^2)_+^{-\frac{a-1}{2}},\quad R = \Big( \frac{\sqrt{\pi}\Gamma(\frac{3-a}{2})\sin\big((a-1)\frac{\pi}{2}\big)}{\Gamma(\frac{4-a}{2})(a-1)\pi}\Big)^{\frac{1}{a-2}}\,,
	\end{equation}
	obtained by \cite{Fra}.
 \end{itemize}
Explicit minimizers for the anisotropic power-law potentials $W(\bx)=\frac{|\bx|^2}{2}-\frac{|\bx|^b}{b}\Omega(\frac{\bx}{|\bx|})$, where $\Omega:S^{d-1}\rightarrow \mathbb{R}_+$ is an angular profile, have also been studied extensively \cite{MRS,CMMRSV1,CMMRSV2,MMRSV1,MMRSV2,CS_2D,CS_3D}. Under suitable assumptions on $\Omega$, the unique minimizer is given by a suitable rescaling and rotation of \eqref{Rx2} or its degenerate analogs.
 
 In this paper we obtain some new explicit formulas for the energy minimizers with power-law potentials, described as follows:
 \begin{itemize}
 	\item For odd $d$ with $(a,b)=(3,2-d)$, the unique minimizer is given as in Theorem \ref{thm_odd} (see the sentence below it for a complete description of the minimizer). Although the general formula is lengthy and involves power series expressions, the minimizers in $d=1$ and $d=3$ can be expressed by elementary functions. For $d=1$, the minimizer is given by
 	\begin{equation}
 		\rho_1(x)=C\cosh(\sqrt{2}x)\chi_{|x|\le R}\,,
 	\end{equation}
 	where $R_1\approx 0.848300901770900$ is the unique positive number satisfying
 	\begin{equation}
 		\sqrt{2}R = \coth(\sqrt{2}R)\,.
 	\end{equation}
 	For $d=3$, the minimizer is given by (denoting $r=|\bx|$, $\beta = 2^{1/4}$)
 	\begin{equation}\begin{split}
 		\rho_3(\bx) = \frac{C}{r} \Big[ & \big(\cosh(\beta r)\sin(\beta r) + \sinh(\beta r)\cos(\beta r)\big) \\
 		& +  \tan(\beta R)\tanh(\beta R)\big(\cosh(\beta r)\sin(\beta r) - \sinh(\beta r)\cos(\beta r)\big)\Big]\chi_{r\le R}\,,
 	\end{split}\end{equation}
 	where $R_3\approx 0.921238965647461$ is the unique positive number satisfying
	\begin{equation}
		\cos(2\beta R)+\cosh(2\beta R)+\beta R\sin(2\beta R)-\beta R \sinh(2\beta R) = 0\,.
	\end{equation} 
 	\item For even $d$ with $(a,b)=(3,1-d)$, the unique minimizer is given as in Corollary \ref{cor_even}, which is a rescaling of the projection of the minimizer in $d+1$ dimensions with the same $(a,b)$. In particular, for the case $d=2$, the minimizer is given by
 	\begin{equation}
 		\rho_2(x_1,x_2) = \lambda^2 \int_{\mathbb{R}}\rho_3(\lambda x_1,\lambda x_2,x_3)\rd{x_3},\quad \lambda = \Big(\frac{3}{8}\Big)^{1/4}\,,
 	\end{equation}
 	which is supported on $B(0;R_2),\,R_2 = \frac{R_3}{\lambda}\approx 1.177238568926828$.
 \end{itemize}
 The explicit minimizers we obtain in dimensions $1,2,3$ are illustrated in Figure \ref{fig1}.
 
 \begin{figure}[htp!]
 	\begin{center}
 		\includegraphics[width=0.32\textwidth]{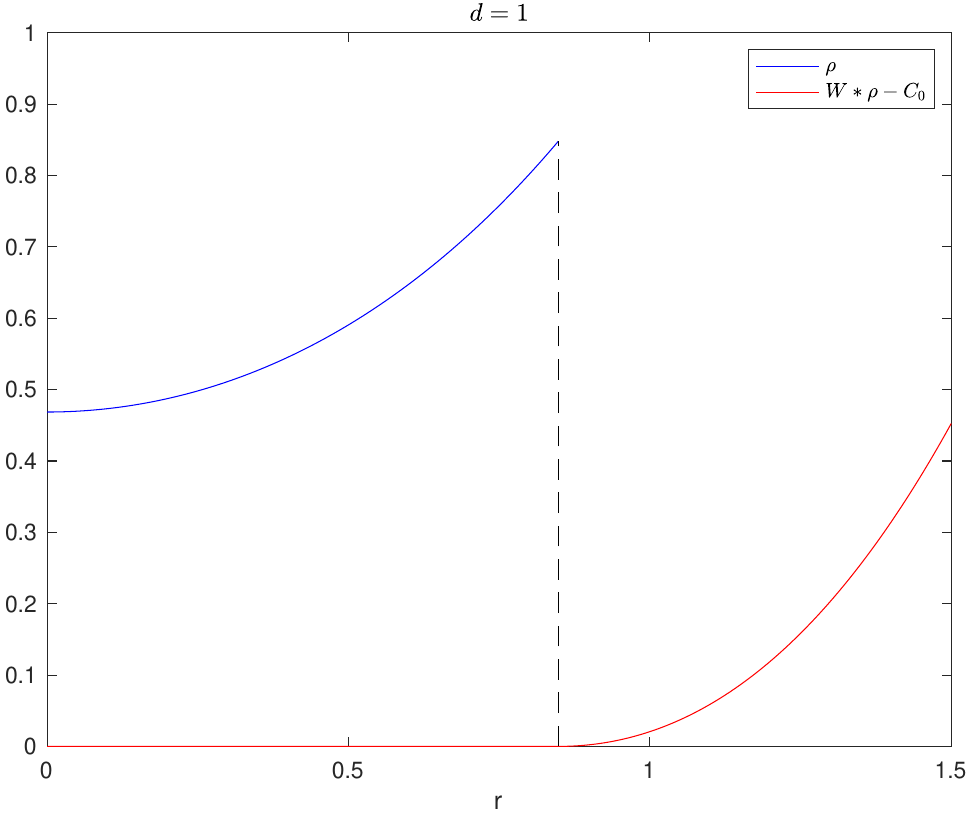}
 		\includegraphics[width=0.323\textwidth]{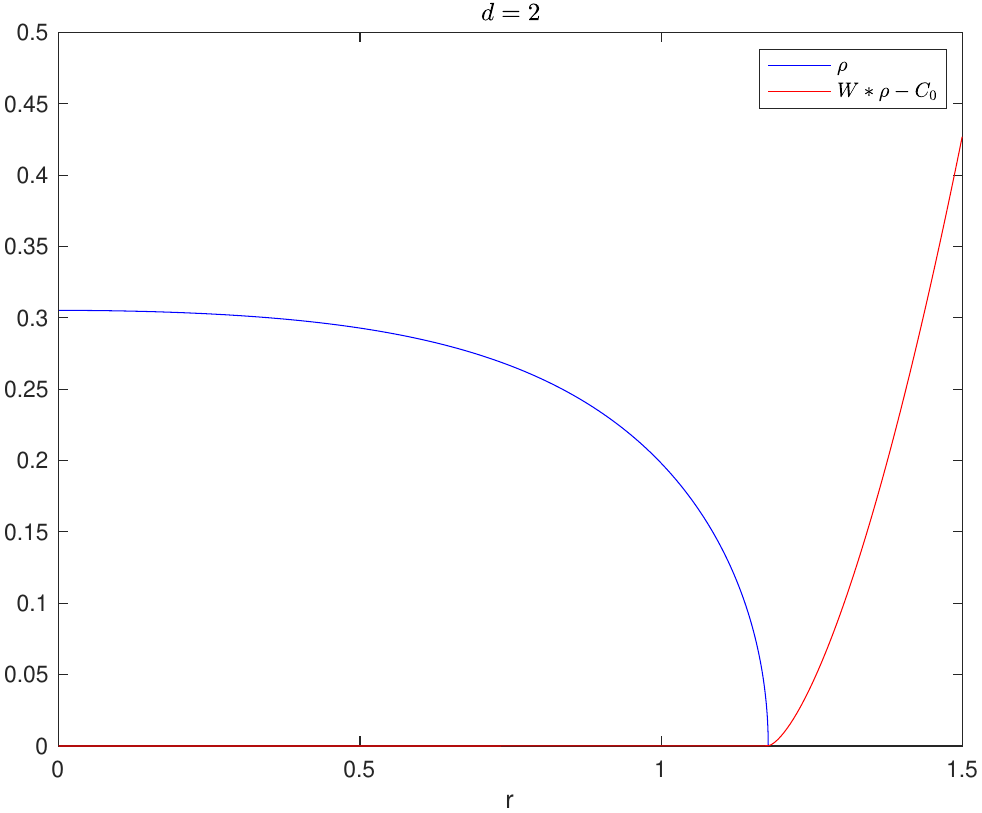}
 		\includegraphics[width=0.32\textwidth]{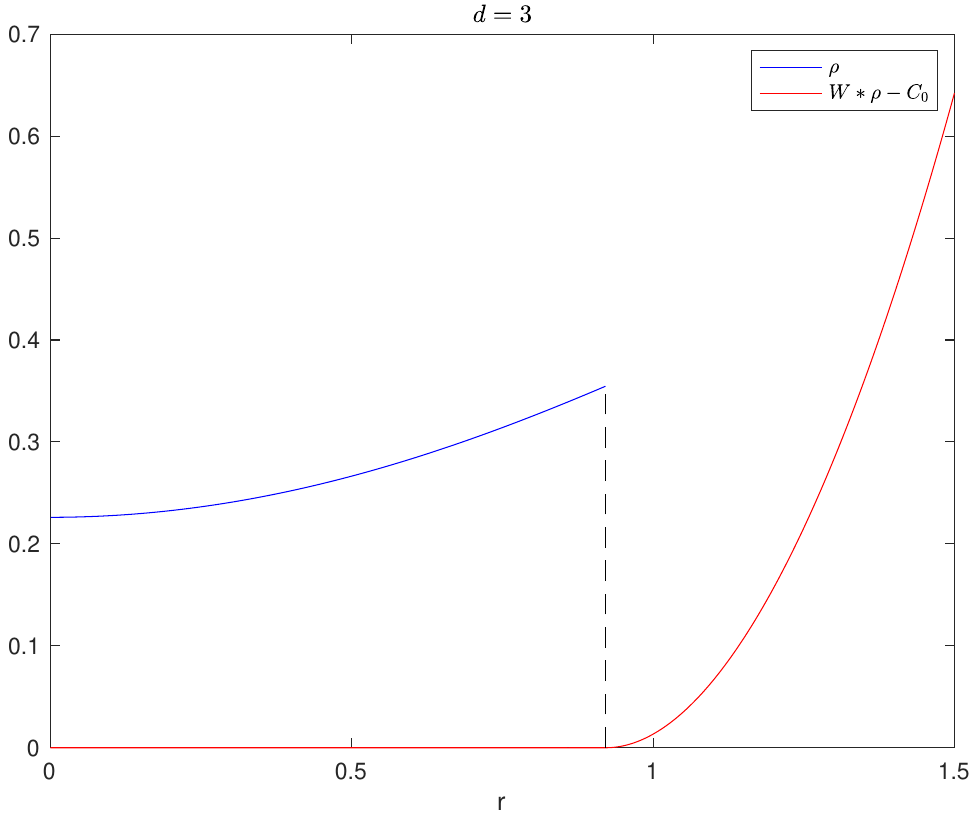}
 		\caption{The minimizers for $(d,a,b)$ being $(1,3,1)$, $(2,3,-1)$ and $(3,3,-1)$ (left, middle and right, respectively). The blue curves are the minimizers $\rho$ in the radial coordinate. The red curves are the generated potentials $W*\rho$, subtracted by the constant $C_0=\frac{1}{2}E[\rho]$.}
 		\label{fig1}
 	\end{center}	
 \end{figure}

 Since the potentials we are considering are LIC, it suffices to verify the Euler-Lagrange condition in order to justify the explicit minimizers. For the odd dimension cases, the key observation is that the Euler-Lagrange condition $W*\rho=C_0$ on $\supp\rho$, after taking successive Laplacian, becomes a local partial differential equation in $\rho$, namely, $\Delta^{(d+1)/2}\rho=A_*\rho$ for some constant $A_*$, see Section \ref{sec_odd}. This equation can be solved in the radial coordinate by power series. In this way we construct a steady state for the energy, which can be shown to be the unique minimizer via qualitative analysis in Section \ref{sec_qual}. For the even dimension case, we first give a new relation, Lemma \ref{lem_dim}, between steady states in dimensions $d$ and $d+1$ via projection. Then we apply it to power-law potentials in Section 5 to get the explicit minimizers in even dimensions.

In principle our technique can also produce `steady states' for the following two cases: (1) odd $d$ with $a-d$ and $b-d$ being even integers; (2) even $d$ with $a-d$ and $b-d$ being odd integers. Here, by `steady state' we mean a radial function $f$ supported on some ball such that $W*f=C$ on $\supp f$. However this may not always give an energy minimizer because $f$ may not be nonnegative. Even if $f$ is nonnegative, we do not know whether $f$ is a minimizer if LIC does not hold. The derivation and realistic meaning of such `steady states' deserve further investigation.

\section{Qualitative results}\label{sec_qual}

In this section we summarize and extend the known qualitative results about energy minimizers with $W$ given by \eqref{Wab}.

\begin{proposition}\label{prop_qual}
	Consider $E$ in \eqref{E} with $W$ given by \eqref{Wab} with $-d<b\le 2 \le a \le 4,\,b<a$. There exists a unique minimizer $\rho_\infty$ of $E$ up to translation. $\rho_\infty$ is compactly supported and radially symmetric. It is the only compactly supported probability measure satisfying the Euler-Lagrange condition
	\begin{equation}\label{prop_qual_1}
		W*\rho = C_0,\quad \rho\textnormal{-a.e.};\quad W*\rho \ge C_0,\quad \textnormal{a.e. }\mathbb{R}^d\,,
	\end{equation}
	for some $C_0\in\mathbb{R}$.
	
	If one further assumes $-d<b\le 2-d$, then $\rho_\infty$ is an $L^\infty$ function that is bounded from below on its support when $b=2-d$, and a $C^\alpha$ function (for any $0<\alpha<1-\frac{b+d}{2}$) when $-d<b<2-d$. $W*\rho_\infty$ is a $C^1$ function.  $\supp\rho_\infty$ is a ball. Furthermore, $\rho_\infty$ is the only radial function in the same regularity class, support being a ball, satisfying $W*\rho = C_0$ on $\supp\rho$ for some $C_0$.
\end{proposition}

\begin{proof}
	We first assume $-d<b\le 2 \le a \le 4,\,b<a$. The existence of minimizer and its compact support property follow from \cite[Theorem 1.4]{CCP15}. The uniqueness follows from the LIC property of $W$, which can be found in \cite{Lop19,CS21,Fra}, and the radial symmetry follows from uniqueness. The fact that the minimizer satisfies the Euler-Lagrange condition \eqref{prop_qual_1} follows from \cite[Theorem 4]{BCLR13_1}. The fact that the minimizer is the only such compactly supported probability measure follows from \cite[Theorem 2.4]{CS21}.
	
	Now we further assume $-d<b\le 2-d$. Then the regularity of the minimizer $\rho_\infty$ and $W*\rho_\infty$ follows from \cite[Theorems 3.4 and 3.10]{CDM} for $b=2-d$ and $-d<b<2-d$, respectively. For the case $b=2-d$, \cite[Theorem 3.4]{CDM} also gives the relation $\Delta\frac{|\bx|^a}{a}*\rho = |S^{d-1}|\rho$ on $\supp\rho$, which implies that $\rho$ is bounded from below on its support since $\Delta\frac{|\bx|^a}{a}$ is nonnegative, continuous, and only possibly vanishes at a single point $\bx=0$.
	
	To see that $\supp\rho_\infty$ is a ball, we only need to show that $(\supp\rho_\infty)^c$ has no bounded connected component. Assume the contrary that $S$ is such a component, then $\partial S\subset \supp\rho_\infty$, and thus $W*\rho_\infty=C_0$ on $\partial S$ by continuity of $W*\rho_\infty$. We also have $\Delta (W*\rho_\infty) > 0$ in $S$ because $\Delta W>0$ on $\mathbb{R}^d\backslash \{0\}$ for $-d<b\le 2-d$ and $ 2 \le a \le 4$. Therefore we apply the maximal principle to $W*\rho_\infty$ on $S$ to see that $\inf_S W*\rho_\infty < C_0$, contradicting \eqref{prop_qual_1}.
	
	To prove the last sentence in the statement, we assume $\rho$ is a radial function in the same regularity class, support being a ball, say, $\overline{B(0;R)}$, satisfying $W*\rho = C_0 $ on $\supp\rho$ for some $C_0$. It follows easily that $W*\rho$ is a $C^1$ function, and we have $\Delta (W*\rho) > 0$ on $\overline{B(0;R)}^c$. Denoting $(W*\rho)(\bx) = f(|\bx|)$, we switch to the radial coordinate and get
	\begin{equation}
		r^{1-d}\frac{\rd}{\rd{r}}\Big(r^{d-1}\frac{\rd{f}}{\rd{r}}\Big) > 0,\quad \text{on }(R,\infty)\,,
	\end{equation}
	with boundary condition $f(R)=C_0,\,f'(R)=0$ since $W*\rho$ is $C^1$. Then it follows that $f(r)>C_0$ for $r>R$, which gives \eqref{prop_qual_1}. Therefore $\rho$ is the minimizer.
\end{proof}

\section{The cases of odd $d$ with $(a,b)=(3,2-d)$}\label{sec_odd}

In this section we treat the odd dimensions with $b=2-d$ and $a=3$. We first study the case $d=1$ to illustrate the procedure (whose rigorous justification is delayed to the proof of Theorem \ref{thm_odd}), and then proceed to the general case.

\subsection{The case $d=1$}\label{sec_d1}

For $d=1$, the unique minimizer $\rho$ is an even function supported on $[-R,R]$ for some $R>0$. The Euler-Lagrange condition gives
\begin{equation}\label{ELd1b1_1}
	 \frac{|x|^3}{3}*\rho - |x|*\rho = C_0,\quad \text{on }(-R,R)\,.
\end{equation}
Taking Laplacian gives
\begin{equation}\label{ELd1b1_2}
	2|x|*\rho - 2\rho = 0,\quad \text{on }(-R,R)\,,
\end{equation}
and taking another Laplacian gives
\begin{equation}
	4\rho - 2\rho'' = 0,\quad \text{on }(-R,R)\,,
\end{equation}
i.e.,
\begin{equation}
	\rho'' = 2\rho,\quad \text{on }(-R,R)\,.
\end{equation}
Solving this equation with the condition $\rho'(0)=0$ (from the even property of $\rho$) gives
\begin{equation}\label{u2k_d1}
	\rho(x) = c_0\cosh(\sqrt{2}x),\quad \text{on }(-R,R)\,,
\end{equation}
where $c_0>0$. To determine $R$, we use \eqref{ELd1b1_2} at $x=0$ and get
\begin{equation}
	c_0 = \int_{-R}^R |x|\rho(x)\rd{x} = 2c_0 \int_0^R x \cosh(\sqrt{2}x) \rd{x} = c_0\Big( \sqrt{2}R\sinh(\sqrt{2}R) - \cosh(\sqrt{2}R) + 1 \Big)\,,
\end{equation}
i.e.,
\begin{equation}\label{detM_d1}
	\sqrt{2}R = \coth(\sqrt{2}R)\,.
\end{equation}
This uniquely determines $R>0$ because $\coth(\cdot)$ is decreasing on $(0,\infty)$ with $\lim_{t\rightarrow 0^+}\coth(t)=\infty$ and $\lim_{t\rightarrow \infty}\coth(t)=1$. Finally one chooses the normalization constant $c_0=\big(\int_{-R}^R\cosh(\sqrt{2}x)\rd{x}\big)^{-1}$ and gets the minimizer.

\subsection{The general case}

For a general odd dimension $d$, the unique minimizer $\rho$ is a radial function supported on $\overline{B(0;R)}$ for some $R>0$. The Euler-Lagrange condition gives
\begin{equation}\label{ELg}
	\frac{|\bx|^3}{3}*\rho - \frac{|\bx|^{2-d}}{2-d}*\rho = C_0,\quad \text{on }B(0;R)\,.
\end{equation}
Taking Laplacian gives
\begin{equation}\label{ELg1}
	(d+1)|\bx|*\rho - |S^{d-1}|\rho = 0,\quad \text{on }B(0;R)\,.
\end{equation}
Then taking Laplacian $k$ times, $k=1,2,\dots,\frac{d-1}{2}$, gives
\begin{equation}\label{ELgk}
	A_{2k}|\bx|^{1-2k}*\rho - |S^{d-1}|\Delta^k\rho = 0,\quad \text{on }B(0;R)\,,
\end{equation}
where
\begin{equation}\label{A2k}
	A_{2k} = (d+1) \big(1\cdot (-1) \cdots \cdot (3-2k)\big) \big((d-1)(d-3)\cdots(d-2k+1)\big)\,.
\end{equation}
Notice that \eqref{ELg1} can be viewed as \eqref{ELgk} with $k=0$. Taking one more Laplacian gives
\begin{equation}\label{ELg2}
	\Delta^{(d+1)/2}\rho = A_{d-1}(2-d)\rho,\quad \text{on }B(0;R)\,.
\end{equation}
Then, we switch to radial coordinate $\rho(\bx) = u(|\bx|)$ and get
\begin{equation}
	\Big(\frac{\rd^2}{\rd r^2} + \frac{d-1}{r}\frac{\rd}{\rd r}\Big)^{(d+1)/2} u = A_{d-1}(2-d)u ,\quad \text{on } (0,R)\,.
\end{equation}
We seek for the power-series solution
\begin{equation}
	u(r) = \sum_{n=0}^\infty c_{2n} r^{2n}\,,
\end{equation}
on $r\in (0,R)$. To do this, we calculate
\begin{equation}
	\Big(\frac{\rd^2}{\rd r^2} + \frac{d-1}{r}\frac{\rd}{\rd r}\Big)u = \sum_{n=0}^\infty c_{2n+2}(2n+2)(2n+d) r^{2n}\,,
\end{equation}
and thus
\begin{equation}\begin{split}
	\Big( \frac{\rd^2}{\rd r^2} + & \frac{d-1}{r}\frac{\rd}{\rd r}\Big)^{(d+1)/2}u \\ = & \sum_{n=0}^\infty c_{2n+d+1}\big((2n+2)(2n+4)\cdots(2n+d+1)\big)\\ & \cdot \big((2n+d)(2n+d+2)\cdots(2n+2d-1)\big) r^{2n}\,.
\end{split}\end{equation}
This gives the iterative relation
\begin{equation}
	c_{2n+d+1} = B_{2n} c_{2n}\,,
\end{equation}
where
\begin{equation}\label{B2n}
	B_{2n} = A_{d-1}(2-d)\Big((2n+2)(2n+4)\cdots(2n+d+1)\Big)^{-1}\Big((2n+d)(2n+d+2)\cdots(2n+2d-1)\Big)^{-1}.
\end{equation}
Therefore, combining terms in the power series of $u$ connected via this relation, we get
\begin{equation}\label{u2k}
	u(r) = \sum_{k=0}^{(d-1)/2}c_{2k} u_{2k}(r),\quad u_{2k}(r) = \sum_{n=0}^\infty \Big(\prod_{i=0}^{n-1} B_{i(d+1)+2k}\Big) r^{n(d+1)+2k} \,,
\end{equation}
where $c_{2k},\,k=0,\dots,\frac{d-1}{2}$ are $\frac{d+1}{2}$ undetermined coefficients. It is clear that the convergence radius of the power series for each $u_{2k}$ is infinity, and thus it defines a real analytic function in $r\in\mathbb{R}$. The minimizer $\rho$ will be given as
\begin{equation}\label{rhoRc}
	\rho[R;\vec{c}](\bx) = \sum_{k=0}^{(d-1)/2}c_{2k} u_{2k}(|\bx|) \chi_{|\bx|\le R}\,,
\end{equation}
where $\vec{c}=(c_0,c_2,\dots,c_{d-1})^\top$. 

Then we use \eqref{ELgk} with $k=0,\dots,\frac{d-1}{2}$ at 0 to obtain the conditions on these coefficients. We calculate
\begin{equation}\label{D2k}
	(\Delta^k \rho)(0) = c_{2k}D_{2k},\quad D_{2k}=\big((2k)(2k-2)\cdots 2\big)\big((2k+d-2)(2k+d-4)\cdots d\big)\,,
\end{equation}
and
\begin{equation}\begin{split}
	(|\bx|^{1-2k}*\rho)(0) = & \int_{B(0;R)}|\bx|^{1-2k}\rho(\bx)\rd{\bx} = |S^{d-1}|\int_0^R r^{d-2k}u(r)\rd{r} \\
	= & |S^{d-1}|\sum_{j=0}^{(d-1)/2}c_{2j}\int_0^R r^{d-2k}u_{2j}(r)\rd{r} \\
	= & |S^{d-1}|\sum_{j=0}^{(d-1)/2}c_{2j}\sum_{n=0}^\infty \Big(\prod_{i=0}^{n-1} B_{i(d+1)+2j}\Big)\frac{R^{n(d+1)+2j+d-2k+1}}{n(d+1)+2j+d-2k+1}\,.
\end{split}\end{equation}
Therefore, equating $A_{2k}(|\bx|^{1-2k}*\rho)(0)$ and  $|S^{d-1}|(\Delta^k\rho)(0)$, we get a homogeneous system of linear equations in $\vec{c}$, where the coefficient matrix $M(R)=(m_{kj}(R))$ is $\frac{d+1}{2}\times \frac{d+1}{2}$, given by
\begin{equation}\label{mkj}\begin{split}
	m_{kj}(R) = A_{2k}\sum_{n=0}^\infty \Big(\prod_{i=0}^{n-1} B_{i(d+1)+2j}\Big)\frac{R^{n(d+1)+2j+d-2k+1}}{n(d+1)+2j+d-2k+1}& - D_{2k}\delta_{kj},\\ & k,j=0,\dots,(d-1)/2\,,
\end{split}\end{equation}
with each $m_{kj}(R)$ being a function of $R$ expressed in a power series.

To admit a nontrivial solution to this system of linear equations, one needs
\begin{equation}\label{detM}
	\det M(R) = 0\,,
\end{equation}
which is a nonlinear equation in $R$. For any $R$ satisfying this equation, one can obtain solutions to $c_{2k},\,k=0,\dots,\frac{d-1}{2}$, and thus radial functions $\rho$ satisfying the Euler-Lagrange condition \eqref{ELg}. Here we remark that for the case $d=1$, \eqref{u2k} reduces to \eqref{u2k_d1}, and \eqref{detM} reduces to \eqref{detM_d1}.

Then we rigorously justify this procedure by showing that it yields the energy minimizer.
\begin{theorem}\label{thm_odd}
	Let $d$ be odd and $(a,b)=(3,2-d)$. There exists a unique positive solution $R$ to \eqref{detM}, such that the null space of $M(R)$ contains a nonzero vector $\vec{c}$ such that $\rho[R,\vec{c}]$ is nonnegative on $B(0;R)$. Furthermore, the null space of this $M(R)$ is one-dimensional. $\rho[R,\vec{c}]$, where $\vec{c}$ is a spanning vector of this null space, after normalization, is the unique minimizer of $E$ up to translation.
\end{theorem}

We remark that all the information needed to determine the minimizer $\rho$ is contained in the formulas \eqref{A2k}, \eqref{B2n}, \eqref{u2k}, \eqref{rhoRc}, \eqref{D2k}, \eqref{mkj} and \eqref{detM}.

\begin{proof}
	Due to Proposition \ref{prop_qual}, the unique minimizer $\rho$ of $E$ satisfies \eqref{ELg}. The above calculation shows that $\rho$ satisfies \eqref{ELg2} in the sense of distribution. Since we already know that $\rho\in L^\infty$ from Proposition \ref{prop_qual}, \eqref{ELg2} implies that $\rho$ is smooth in $B(0;R)$, and thus can be expressed as in \eqref{rhoRc}, with its support size $R$ and the corresponding $\vec{c}$ satisfying $M(R)\vec{c} = 0$. Therefore this $R$ solves \eqref{detM} and $\vec{c}$ is a nonzero vector in the null space of $M(R)$ whose corresponding $\rho[R;\vec{C}]=\rho$ is nonnegative on $B(0;R)$. Therefore the existence of $R$ as described is obtained.
	
	Then we prove the uniqueness of $R$. Suppose $R$ satisfies the properties as described, then we take a nonzero vector $\vec{c}$ in the null space of $M(R)$ such that $\rho[R;\vec{c}]$ is nonnegative on $B(0;R)$. Since $M(R)\vec{c}=0$, \eqref{ELgk} is satisfied at 0 for $k=0,1,\dots,\frac{d-1}{2}$. Since \eqref{ELgk} for these $k$ are successive Laplacian of \eqref{ELg}, one can use induction to show that \eqref{ELgk} for these $k$ are satisfied. As a consequence, \eqref{ELg} is satisfied. After normalization, $\rho[R;\vec{c}]$ is a probability measure which is also a radial $L^\infty$ function supported on $\overline{B(0;R)}$, satisfying the Euler-Lagrange condition \eqref{ELg}. Therefore Proposition \ref{prop_qual} (the last sentence in its statement) implies that $\rho[R;\vec{c}]$ is the unique minimizer of $E$. This proves the uniqueness of $R$.
	
	 To see that the null space of the $M$ corresponding to the prescribed $R$ is 1-dimensional, we assume the contrary. Then there exists a vector $\vec{c}\,'$ which is linearly independent to $\vec{c}$, the coefficient vector corresponding to the minimizer $\rho$. Let $\rho'=\rho[R;\vec{c}\,']$, then $\rho'$ is a radial $L^\infty$ function supported on $\overline{B(0;R)}$ that also satisfies \eqref{ELg}. Construct
	 \begin{equation}
	 	\rho_\epsilon := \Big(1-\epsilon \int_{B(0;R)}\rho'\rd{\bx}\Big)\rho + \epsilon\rho'\,,
	 \end{equation}
	 which satisfies $\int_{B(0;R)}\rho_\epsilon\rd{\bx}=1$. Notice that $\rho$ is bounded from below on $B(0;R)$ by Proposition \ref{prop_qual}, and $\rho'$ is in $L^\infty$. Therefore, for sufficiently small $\epsilon>0$, $\rho_\epsilon$ is a probability measure which is also a radial $L^\infty$ function supported on $\overline{B(0;R)}$, satisfying \eqref{ELg}, and thus has to be the unique minimizer of $E$, which is $\rho$. This leads to a contradiction since $\rho'$ is not a constant multiple of $\rho$ due to the linear independence of the corresponding coefficients.
\end{proof}

\begin{remark}
	For $d=1$ and $d=3$ we know that the positive solution $R$ to \eqref{detM} is unique, see \eqref{detM_d1} and \eqref{detM_d3} and the discussion afterwards. However, we do not know whether the same uniqueness holds in higher dimensions. Therefore, to numerically calculate the minimizer via $R$ and $\vec{c}$, one needs to find all positive solutions $R$ to \eqref{detM}, and identifies the one with the correct null space.
\end{remark}

\subsection{The case $d=3$ in elementary functions}\label{sec_d3}

For the case $d=3$, the functions $u_0(r)$ and $u_2(r)$ can be expressed by elementary functions as
\begin{equation}
	u_0(r) = \frac{1}{2\beta r}\big(\cosh(\beta r)\sin(\beta r) + \sinh(\beta r)\cos(\beta r)\big),\quad \beta = 2^{1/4}\,,
\end{equation}
and
\begin{equation}
	u_2(r) = \frac{3}{2\beta^3 r}\big(\cosh(\beta r)\sin(\beta r) - \sinh(\beta r)\cos(\beta r)\big)\,.
\end{equation}
Then we have
\begin{equation}\label{detM_d3}
	\det M(R) = 3\big(\cos(2\beta R)+\cosh(2\beta R)+\beta R\sin(2\beta R)-\beta R \sinh(2\beta R)\big)\,.
\end{equation}
The function $\det M(R)$ is decreasing on $[0,\infty)$ because 
\begin{equation}
	\phi(t) = \frac{1}{3}\det \Big(M\big(\frac{t}{\beta}\big)\Big) = \cos2t + \cosh2t + t\sin2t - t\sinh2t\,,
\end{equation} 
satisfies $\phi'(0)=0$ and $\phi''(t)=-4t(\sin2t + \sinh2t) < 0$ for $t>0$. Then, since $\det M(0) = 6 > 0$ and $\lim_{R\rightarrow\infty}\det M(R) = -\infty$, we see that $\det M(R)=0$ has a unique solution. Then we use the linear system in $\vec{c}$ to get
\begin{equation}
	\frac{c_2}{c_0} = \frac{\sqrt{2}}{3} \tan(\beta R)\tanh(\beta R)\,,
\end{equation}
and then $c_0$ can be obtained by normalization.

\section{Connection between steady states in different dimensions}

In this section we study the relation between the energy minimizers or steady states in dimensions $d$ and $d+1$. We first introduce two operators. Denote $L^1_{\textnormal{rad},\textnormal{loc}}(\mathbb{R}^d)$ as the space of locally integrable radial functions on $\mathbb{R}^d$,  and $L^1_{\textnormal{rad},R}(\mathbb{R}^d)$ as its subspace containing functions supported on $\overline{B(0;R)}$. For fixed $R>0$, define the operator $\cP: L^1_{\textnormal{rad},R}(\mathbb{R}^{d+1})\rightarrow L^1_{\textnormal{rad},R}(\mathbb{R}^d)$ by
\begin{equation}
	\cP[F](x_1,\dots,x_d) = \int_{\mathbb{R}} F(x_1,\dots,x_d,x_{d+1})\rd{x_{d+1}}\,.
\end{equation}
Define the operator $\cQ: L^1_{\textnormal{rad},\textnormal{loc}}(\mathbb{R}^d)\rightarrow L^1_{\textnormal{rad},\textnormal{loc}}(\mathbb{R}^{d+1})$ by
\begin{equation}
	\cQ[\tilde{W}](\bx) =\frac{1}{|S^d|} \int_{S^d}\tilde{w}\big(\dist(\langle \bv \rangle, \bx)\big)\rd{S(\bv)}\,,
\end{equation}
where $\tilde{w}(|\bx|)=\tilde{W}(\bx)$, and $\langle \bv \rangle$ denotes the subspace of $\mathbb{R}^{d+1}$ spanned by $\bv$. It is clear that $\cQ$ is well-defined. In fact, for $\tilde{W}\in L^1_{\textnormal{rad},\textnormal{loc}}(\mathbb{R}^d)$, if $\bv=(0,\dots,0,1)\in S^d$, then the function $\bx\mapsto\tilde{w}\big(\dist(\langle \bv \rangle, \bx)\big) = \tilde{W}(x_1,\dots,x_d)$ is a locally integrable function defined on $\mathbb{R}^{d+1}$. The same is true for general $\bv\in S^d$ due to rotational symmetry. It follows that $\cQ[\tilde{W}]\in L^1_{\textnormal{rad},\textnormal{loc}}(\mathbb{R}^{d+1})$ because it is an average of these locally integrable functions which are rotations of each other.

The ball $B(0;R)$ in $\mathbb{R}^d$ will be denoted as $B_d(0;R)$ to indicate its underlying dimension.
\begin{lemma}\label{lem_dim}
	Let $\tilde{W}\in L^1_{\textnormal{rad},\textnormal{loc}}(\mathbb{R}^d)$, $W=\cQ[\tilde{W}]$, $F\in L^1_{\textnormal{rad},R}(\mathbb{R}^{d+1})$, $\tilde{F}=\cP[F]$. Then $W*F=C$ on $\textnormal{a.e. }B_{d+1}(0;R)$ is equivalent to $\tilde{W}*\tilde{F}=C$ on $\textnormal{a.e. }B_d(0;R)$.
\end{lemma}
\begin{proof}
	Denote $\bx=(x_1,\dots,x_d,x_{d+1})$ and $\tilde{\bx}=(x_1,\dots,x_d)$. For $\bv=(0,\dots,0,1)\in S^d$,  we calculate 
	\begin{equation}\label{tildeWF}\begin{split}
		(\tilde{w}\big(\dist(\langle \bv \rangle, \cdot)\big)*F)(\bx) = & \int_{\mathbb{R}^{d+1}}\tilde{w}(|\tilde{\by}|) F(\bx-\by)\rd{\by} \\
		= & \int_{\mathbb{R}^d}\tilde{w}(|\tilde{\by}|) \int_{\mathbb{R}}F(\tilde{\bx}-\tilde{\by},x_{d+1}-y_{d+1})\rd{y_{d+1}}\rd{\tilde{\by}} \\
		= & \int_{\mathbb{R}^d}\tilde{w}(|\tilde{\by}|) \tilde{F}(\tilde{\bx}-\tilde{\by})\rd{\tilde{\by}} \\
		= &  (\tilde{W}*\tilde{F})(\tilde{\bx})\,.  \\
	\end{split}\end{equation}
	Denoting $\tilde{U}=\tilde{W}*\tilde{F}$ and $\tilde{u}(|\tilde{\bx}|)=\tilde{U}(\tilde{\bx})$, then the last quantity is $\tilde{u}(\dist(\langle \bv \rangle, \bx))$. By rotating the coordinate system, one can show that for any $\bv\in S^d$,
	\begin{equation}
		(\tilde{w}\big(\dist(\langle \bv \rangle, \cdot)\big)*F)(\bx) = \tilde{u}(\dist(\langle \bv \rangle, \bx))\,.
	\end{equation}
	Integrating in $\bv$, we get
	\begin{equation}\label{Qst}
		(W*F)(\bx) = \frac{1}{|S^d|} \int_{S^d}\tilde{u}(\dist(\langle \bv \rangle, \bx))\rd{S(\bv)} = \cQ[\tilde{W}*\tilde{F}](\bx)\,.
	\end{equation}
	
	If $\tilde{W}*\tilde{F}=C$ on $\textnormal{a.e. }B_d(0;R)$, then $\tilde{u}=C$ on $\textnormal{a.e. }[0,R]$. For any $|\bx|\le R$, we have $\dist(\langle \bv \rangle, \bx)\le R$ for any $\bv\in S^d$, and thus the above integrand is $C$ for $\textnormal{a.e. }S^d$ (here to treat the a.e. issue, we also used the fact that any preimage of a zero measure set for $\dist(\langle \cdot \rangle, \bx)$ has surface measure zero on $S^d$). It follows that $W*F=C$ on $\textnormal{a.e. }B_{d+1}(0;R)$.
	
	To prove the other direction, it suffices to show that $\cQ[\tilde{U}]=0$ on $\textnormal{a.e. }B_{d+1}(0;R)$ implies $\tilde{U}=0$ on $\textnormal{a.e. }B_d(0;R)$ for any $\tilde{U}\in  L^1_{\textnormal{rad},\textnormal{loc}}(\mathbb{R}^d)$. In fact, suppose the former, then taking a radial test function $\phi\in L^1_{\textnormal{rad},R}(\mathbb{R}^{d+1})$ satisfying $\cP[\phi]\in L^\infty(B_d(0;R))$, we have
	\begin{equation}\label{PQ}
		\int_{\mathbb{R}^{d+1}} \cQ[\tilde{U}]\phi\rd{\bx} = \int_{\mathbb{R}^d}\tilde{U}\cP[\phi]\rd{\tilde{\bx}}\,.
	\end{equation}
	This can be justified by showing that $\int_{\mathbb{R}^{d+1}}\tilde{u}\big(\dist(\langle \bv \rangle, \bx)\big)\phi(\bx) \rd{\bx} = \int_{\mathbb{R}^d}\tilde{U}\cP[\phi]\rd{\tilde{\bx}}$ for $\bv=(0,\dots,0,1)\in S^d$ similar to \eqref{tildeWF}, and then averaging in $\bv$ using rotational symmetry.
	
	Take 
	\begin{equation}
		\phi_n(\bx) = \frac{|\bx|^{2n}}{\sqrt{R^2-|\bx|^2}}\chi_{B_{d+1}(0;R)},\quad n=0,1,2,\dots\,.
	\end{equation}
	Then for $|\tilde{\bx}|<R$,
	\begin{equation}
		\cP[\phi_n](\tilde{\bx}) = \int_{-\sqrt{R^2-|\tilde{\bx}|^2}}^{\sqrt{R^2-|\tilde{\bx}|^2}}\frac{|\bx|^{2n}}{\sqrt{R^2-|\tilde{\bx}|^2-x_{d+1}^2}}\rd{x_{d+1}}\,.
	\end{equation}
	Notice that $|\bx|^{2n} = (|\tilde{\bx}|^2+x_{d+1}^2)^n = \sum_{k=0}^n{n\choose k}|\tilde{\bx}|^{2(n-k)}x_{d+1}^{2k}$, and we have
	\begin{equation}
		\int_{-\sqrt{R^2-|\tilde{\bx}|^2}}^{\sqrt{R^2-|\tilde{\bx}|^2}}\frac{x_{d+1}^{2k}}{\sqrt{R^2-|\tilde{\bx}|^2-x_{d+1}^2}}\rd{x_{d+1}} = (R^2-|\tilde{\bx}|^2)^{k}\int_{-1}^1 \frac{x^{2k}\rd{x}}{\sqrt{1-x^2}}\,.
	\end{equation}
	Therefore 
	\begin{equation}
	\cP[\phi_n](\tilde{\bx}) = \chi_{B_d(0;R)}\sum_{k=0}^n{n\choose k}\int_{-1}^1 \frac{x^{2k}\rd{x}}{\sqrt{1-x^2}}\cdot |\tilde{\bx}|^{2(n-k)}(R^2-|\tilde{\bx}|^2)^{k}\,.
	\end{equation}
	This is a polynomial of degree at most $n$ in $|\tilde{\bx}|^2$, and in fact its degree is $n$ because the leading coefficient is
	\begin{equation}
		  \sum_{k=0}^n{n\choose k}\int_{-1}^1 \frac{x^{2k}\rd{x}}{\sqrt{1-x^2}} \cdot (-1)^k = \int_{-1}^1 \sqrt{1-x^2}\rd{x} = \frac{\pi}{2}\ne 0\,.
	\end{equation}
	Therefore, $\phi_n$ is an eligible test function. If $\cQ[\tilde{U}]=0$, then it follows from \eqref{PQ} that $\int_{B_d(0;R)}\tilde{U} \cP[\phi_n]\rd{\tilde{\bx}} = 0$, i.e., $\tilde{U}\chi_{B_d(0;R)}$ is orthogonal to all polynomials of $|\tilde{\bx}|^2$. This implies $\tilde{U}=0$ on $\textnormal{a.e. }B_d(0;R)$.
\end{proof}

\begin{remark}
	It is worthwhile to point out the following fact: In Lemma \ref{lem_dim}, if one instead assumes that $\tilde{W}*\tilde{F}\ge C$ on $\textnormal{a.e. }\mathbb{R}^d$, then  $W*F\ge C$ on $\textnormal{a.e. }\mathbb{R}^{d+1}$. This is a direct consequence of \eqref{Qst}. Therefore, the Euler-Lagrange condition \eqref{prop_qual_1} for $(\tilde{W},\tilde{F})$ implies that for $(W,F)$, provided that $W*F$ is continuous.
	
	Although this fact does not play a role in this paper, we believe that it is important and may serve as a tool to find higher dimensional minimizers from lower dimensional ones. In fact, a similar idea was used in \cite{CS_2D,CS_3D} to obtain the explicit minimizers for some anisotropic interaction energies.
\end{remark}

\section{The cases of even $d$ with $(a,b)=(3,1-d)$}

\begin{lemma}\label{lem_Q}
	Let $\tilde{W}(\tilde{\bx})=\frac{|\tilde{\bx}|^a}{a},\,a>-d$ in $\mathbb{R}^d$. Then
	\begin{equation}
		\cQ[\tilde{W}](\bx) = \eta_{d,a} \frac{|\bx|^a}{a},\quad \textnormal{if }a\ne 0\,,
	\end{equation}
	where
	\begin{equation}
		\eta_{d,a} =  \frac{\Gamma(\frac{a+d}{2})\Gamma(\frac{d+1}{2})}{\Gamma(\frac{a+d+1}{2})\Gamma(\frac{d}{2})}\,.
	\end{equation}
	For the case $a=0$, there holds
	\begin{equation}
	\cQ[\ln|\tilde{\bx}|](\bx) = \ln|\bx| + C_{d,\ln},\quad C_{d,\ln} =  \frac{|S^{d-1}|}{|S^d|} \int_0^\pi \sin^{d-1}\phi\ln(\sin\phi)\rd{\phi}\,.
	\end{equation}
\end{lemma}
\begin{proof}
	Denote $\vec{e}_{d+1}=(0,\dots,0,1)\in\mathbb{R}^{d+1}$. Then, if $a\ne 0$, we have
	\begin{equation}\begin{split}
		\cQ[\tilde{W}](r\vec{e}_{d+1}) = & \frac{1}{a|S^d|} \int_{S^d}\big(\dist(\langle \bv \rangle, r\vec{e}_{d+1})\big)^a\rd{S(\bv)} 
		= \frac{r^a}{a|S^d|} \int_{S^d}\big(\dist(\langle \bv \rangle, \vec{e}_{d+1})\big)^a\rd{S(\bv)} \\
		= & \frac{r^a}{a|S^d|} \int_{S^d}\big(\dist( \bv , \langle \vec{e}_{d+1} \rangle)\big)^a\rd{S(\bv)} 
		=  \frac{r^a}{a|S^d|} \int_{S^d}|\tilde{\bv}|^a\rd{S(\bv)} \\
		= & \frac{r^a|S^{d-1}|}{a|S^d|} \int_0^\pi \sin^{a+d-1}\phi\rd{\phi}
		=  \frac{r^a|S^{d-1}|}{a|S^d|}\cdot \frac{\Gamma(\frac{a+d}{2})\Gamma(\frac{1}{2})}{\Gamma(\frac{a+d+1}{2})}\\
		= & \frac{r^a}{a}\cdot \frac{\Gamma(\frac{a+d}{2})\Gamma(\frac{d+1}{2})}{\Gamma(\frac{a+d+1}{2})\Gamma(\frac{d}{2})}\,,\\
	\end{split}\end{equation}
	using the spherical coordinates $\bv = (\bv'\sin\phi, \cos\phi),\,\tilde{\bv} = \bv'\sin\phi$.
	
	For the case $a=0$, a similar calculation gives
	\begin{equation}\begin{split}
		\cQ[\ln|\tilde{\bx}|](r\vec{e}_{d+1}) = & \frac{1}{|S^d|} \int_{S^d}\ln\big(\dist(\langle \bv \rangle, r\vec{e}_{d+1})\big)\rd{S(\bv)} 
		= \ln r + \frac{1}{|S^d|} \int_{S^d}\ln|\tilde{\bv}|\rd{S(\bv)} \\
		= & \ln r + \frac{|S^{d-1}|}{|S^d|} \int_0^\pi \sin^{d-1}\phi\ln(\sin\phi)\rd{\phi}\,.
\end{split}\end{equation}
\end{proof}

\begin{theorem}\label{thm_even}
	Let $-d<b\le 1-d$ and $2\le a\le 4$, $W$ be given by \eqref{Wab} in $\mathbb{R}^{d+1}$, and $\bar{W}$ given by \eqref{Wab} in $\mathbb{R}^d$. Denote $E$ and $\bar{E}$ as the energies associated to $W$ and $\bar{W}$ respectively. Let $\rho\in \cM(\mathbb{R}^{d+1})$ be the unique minimizer of $E$ whose support is $\overline{B_{d+1}(0;R)}$. Then the unique minimizer (up to translation) of $\bar{E}$ is
	\begin{equation}\label{thm_even_1}
		\bar{\rho}(\bx) = \lambda^d\cP[\rho](\lambda \bx),\quad \lambda = \Big(\frac{\eta_{d,a}}{\eta_{d,b}}\Big)^{1/(a-b)}\,,
	\end{equation}
	which is supported on $\overline{B_d(0;\frac{R}{\lambda})}$.
\end{theorem}

\begin{proof}
	We first assume $b\ne 0$. By Proposition \ref{prop_qual}, $W*\rho=C$ on $B_{d+1}(0;R)$. The Lemma \ref{lem_Q} shows that
	\begin{equation}\label{Weta}
		\tilde{W}(\bx) = \eta_{d,a}^{-1}\frac{|\bx|^a}{a}-\eta_{d,b}^{-1}\frac{|\bx|^b}{b}\,,	\end{equation}
	defined on $\mathbb{R}^d$ satisfies $\cQ[\tilde{W}] = W$. Then we apply Lemma \ref{lem_dim} to obtain $\tilde{W}*\cP[\rho]=C$ on $\textnormal{a.e. }B_d(0;R)$. Since $\rho$ is a bounded function which is continuous on $B_{d+1}(0;R)$, it is easy to show that $\cP[\rho]$ is continuous on $\mathbb{R}^d$. This implies the continuity of $\tilde{W}*\cP[\rho]$ on $\mathbb{R}^d$, and thus $\tilde{W}*\cP[\rho]=C$ on $B_d(0;R)$. Then we notice the scaling relation
	\begin{equation}\begin{split}
		(\bar{W}*\bar{\rho})(\lambda^{-1}\bx) = & \int_{\mathbb{R}^d} \Big(\frac{|\lambda^{-1}\bx-\by|^a}{a} - \frac{|\lambda^{-1}\bx-\by|^b}{b}\Big)  \lambda^d\cP[\rho](\lambda \by)\rd{\by} \\
		= & \int_{\mathbb{R}^d} \Big(\lambda^{-a}\frac{|\bx-\by|^a}{a} - \lambda^{-b}\frac{|\bx-\by|^b}{b}\Big)  \cP[\rho]( \by)\rd{\by} \\
		= & \left(\Big(\lambda^{-a}\frac{|\cdot|^a}{a} - \lambda^{-b}\frac{|\cdot|^b}{b}\Big)*  \cP[\rho]\right)(\bx)\,.
	\end{split}\end{equation}
	By choosing $\lambda$ such that
	\begin{equation}
		\lambda^{-a}\eta_{d,b}^{-1} = \lambda^{-b}\eta_{d,a}^{-1}\,,
	\end{equation}
	i.e., the one given by \eqref{thm_even_1}, we see that $\lambda^{-a}\frac{|\cdot|^a}{a} - \lambda^{-b}\frac{|\cdot|^b}{b}$ is a scalar multiple of $\tilde{W}$, and thus the last quantity is constant on $B_d(0;R)$. This implies $\bar{W}*\bar{\rho}$ is constant on $B_d(0;\frac{R}{\lambda})$. Then applying Proposition \ref{prop_qual} (the last sentence in the statement) shows that $\bar{\rho}$ is the minimizer of $\bar{E}$. 
	
	The case $b=0$ can be treated similarly. In fact, we notice that $\eta_{d,0}=1$, and thus Lemma \ref{lem_Q} gives $\cQ[\tilde{W}]=W-C_{d,\ln}$ where $\tilde{W}$ is given by \eqref{Weta}. This implies $\tilde{W}*\cP[\rho]=C-C_{d,\ln}$ on $\textnormal{a.e. }B_d(0;R)$. Then we notice the scaling relation
	\begin{equation}\begin{split}
		(\bar{W}*\bar{\rho})(\lambda^{-1}\bx) = & \int_{\mathbb{R}^d} \Big(\frac{|\lambda^{-1}\bx-\by|^a}{a} - \ln|\lambda^{-1}\bx-\by|\Big)  \lambda^d\cP[\rho](\lambda \by)\rd{\by} \\
		= & \int_{\mathbb{R}^d} \Big(\lambda^{-a}\frac{|\bx-\by|^a}{a} - \ln|\bx-\by| + \ln\lambda\Big)  \cP[\rho]( \by)\rd{\by} \\
		= & \left(\Big(\lambda^{-a}\frac{|\cdot|^a}{a} - \ln|\cdot|\Big)*  \cP[\rho]\right)(\bx) + \ln\lambda\,,
	\end{split}\end{equation}
	and we again obtain that $\bar{W}*\bar{\rho}$ is constant on $B_d(0;\frac{R}{\lambda})$ by the same choice of $\lambda$, and see that $\bar{\rho}$ is the minimizer of $\bar{E}$.
\end{proof}

Then applying Theorem \ref{thm_even} to the case of even $d$ with $(a,b)=(3,1-d)$, we can express the associated energy minimizer as the projection and rescaling of the minimizer in dimension $d+1$ obtained in Theorem \ref{thm_odd}.
\begin{corollary}\label{cor_even}
	Let $d$ be even, $b=1-d$, $a=3$, and $\tilde{E}$ as the energy with potential $\tilde{W}$ given by \eqref{Wab} in $\mathbb{R}^d$. Denote $E$ as the energy with potential $W$ given by \eqref{Wab} in $\mathbb{R}^{d+1}$, and its unique minimizer  $\rho$ given by Theorem \ref{thm_odd}, supported on $\overline{B_{d+1}(0;R)}$. Then the unique minimizer $\tilde{\rho}$ of $\tilde{E}$  (up to translation) is given by \eqref{thm_even_1}, i.e.,
	\begin{equation}
		\tilde{\rho}(x_1,\dots,x_d) = \lambda^d \int_{\mathbb{R}}\rho(\lambda x_1,\dots,\lambda x_d,x_{d+1})\rd{x_{d+1}},\quad \lambda = \Big(\frac{\Gamma(\frac{3+d}{2})}{\Gamma(\frac{4+d}{2})\sqrt{\pi}}\Big)^{1/(d+2)}\,,
	\end{equation}
	which is supported on $\overline{B_d(0;\frac{R}{\lambda})}$.
\end{corollary}

\bibliographystyle{alpha}
\bibliography{minimizer_book_bib.bib}

\end{document}